\documentclass[12pt]{amsart}

\usepackage[utf8]{inputenc}
\usepackage[english]{babel}

\usepackage{mathrsfs}
\usepackage{amsthm}
\usepackage{amsmath}
\usepackage{amsfonts}
\usepackage{mathtools}
\mathtoolsset{showonlyrefs}
\usepackage{amssymb}

\usepackage{textcomp}
\usepackage{xcolor}

\usepackage{calrsfs}
\usepackage{csquotes}

\usepackage{tikz-cd}

\usepackage{comment}

\DeclareMathAlphabet{\mathcal}{OMS}{cmsy}{m}{n}

\theoremstyle{plain}
\newtheorem {theorem}{Theorem}[section]
\newtheorem {lemma}[theorem]{Lemma}
\newtheorem {corollary} [theorem]{Corollary}

\newtheorem {proposition} [theorem]{Proposition}
\theoremstyle{definition}
\newtheorem{definition}[theorem]{Definition}

\theoremstyle{remark}
\newtheorem{remark}[theorem]{Remark}

 \numberwithin{equation}{section} 
 
\newcommand{\R}{\mathbb{R}}

\newcommand{\N}{\mathbb{N}}

\newcommand{\D}{\mathcal{D}}

\newcommand{\oo}{\infty}

\newcommand{\e}{\varepsilon}
\newcommand{\s}{\sigma}

\newcommand{\w}{\omega}

\renewcommand{\d}{\partial}
\renewcommand{\a}{\alpha}

\renewcommand{\k}{\kappa}
\newcommand{\g}{\gamma}

\newcommand{\de}{\delta}
\newcommand{\De}{\Delta}

\newcommand{\cut}{\mathrm{cut}}
\newcommand{\cor}{\mathrm{cor}}

\newcommand{\vspan}{\mathrm{span}}

\newcommand{\mc}{\mathcal}
\newcommand{\mr}{\mathrm}

\newcommand{\wt}{\widetilde}

\newcommand{\be}{\begin{equation}}
\newcommand{\ee}{\end{equation}}

\renewcommand{\k}{\mathrm k}

\usepackage{tikz}
\usetikzlibrary{calc}
\usepackage{pgfplots}
\usetikzlibrary{decorations.markings}
\pgfplotsset{compat=1.18}

\renewcommand{\k}{\kappa}

\newcommand{\dsy}{\displaystyle}

\theoremstyle{definition}

\usepackage{float}

\title{Sharp regularity of sub-Riemannian length-minimizing curves}

\begin{document}

\author{A. Socionovo}

\email{alessandro.socionovo@ensta.fr, alessandro.socionovo@sorbonne-universite.fr}
\address{Laboratoire Jacques-Louis Lions, Sorbonne Universit\'{e}, Universit\'{e} de Paris, CNRS, Inria, Paris, France}
\address{Unit\'{e} de Math\'ematiques Appliqu\'ees (UMA), ENSTA, Institut Polytechnique de Paris, 91120 Palaiseau, France}

\maketitle

\begin{abstract} 
	A longstanding open question in sub-Riemannian geometry is the $C^\infty$-smoothness of length-minimizing curves (in their arc-length parameterization). A recent example answered this question negatively, showing the existence of a sub-Riemannian manifold with a length minimizer of class $C^2\setminus C^3$. In this paper, we study a broader class of sub-Riemannian structures containing that of the aforementioned example, and we prove that length-minimizing curves are of class $C^2$ within these structures. In particular, the aforementioned example is sharp (within this class of sub-Riemannian structures).
\end{abstract}

\tableofcontents

\section{Introduction}

One of the most difficult and important open problems in sub-Riemannian geometry is the regularity of geodesics, i.e., of those curves that are length-minimizing for short times (here and hereafter, by regularity of a curve we mean the regularity of its arc-length parameterization). The problem has remained open since the works of Strichartz in~\cite{Str86, Str89}.

A sub-Riemannian structure is a pair $(M,\D)$, where $M$ is a smooth manifold, $\D\subset TM$ is a smooth bracket-generating distribution (i.e., a distribution satisfying the H\"ormander condition). When $\D$ is endowed with a smooth metric $g$, we have a sub-Riemannian manifold $(M,\D,g)$. A {\em horizontal} (or, {\em admissible}) curve is an absolutely continuous trajectory in $M$ which is tangent to $\D$ almost everywhere. The length of horizontal curves is defined in the standard way with respect to the metric $g$ and the distance between two points $p,q\in M$ is the infimum of the length of horizontal curves joining $p$ and $q$. The Chow-Rashewskii and the Hopf-Rinow theorems ensure the existence of geodesics between every pair of points sufficiently close. 

Candidate curves to be length-minimizing are called {\em extremal curves} and they are individuated by some necessary conditions appearing in the Pontryagin maximum principle. Extremal curves are divided in two non-disjoint classes, which are called {\em normal} and {\em abnormal} (or {\em singular}) curves. Normal curves, which are the only ones appearing in Riemannian geometry, are smooth and length-minimizing for short times. Instead, abnormal curves are a priori no more regular than being absolutely continuous, and their length-minimality properties are also unknown. The difficulty of the problem of the regularity of length-minimizing curves lies in the presence of {\em strictly abnormal} (or {\em strictly singular}) extremals, i.e., those curves that are abnormal but not normal. We refer the reader to~\cite{ABB20,AgrSac,J14,Mon02,Rif14} for an exhaustive introduction to sub-Riemannian geometry and to the regularity problem of geodesics.

The first example of a strictly abnormal curve that is length-minimizing, was discovered by Montgomery in his ground-breaking work~\cite{Mon94}. After that, many other examples have been found, see for instance~\cite{AS96, LS95}. All these examples are smooth curves (here and hereafter, by smooth we mean $C^\oo$), indicating the profound complexity of the problem. Over the years, some partial results have been found. On the one hand, it was first proved in \cite{HL16, LM08} that abnormal length-minimizing curves cannot have corners, and later in \cite{HL23, MPV18} that they must have a tangent line. In some special cases they are of class $C^1$, see \cite{BCJPS20, BFPR22, LPS24}, but in general the $C^1$ regularity remains an open question. On the other hand, necessary conditions for the minimality of abnormal curves can be derived by the differential analysis of the end-point map (see \cite[Chapter 8]{ABB20} for a detailed introduction to the end-point map). This theory is well known up to the second order, see \cite{AgrSac, Goh66}, and it has been partially extended to the $n$-th order in \cite{BMS24}. Finally, an algorithm to produce abnormal curves in Carnot groups of sufficiently high dimension (without any new regularity or minimality result) is given in \cite{H20}.

A recent and significant progress for the theory is the following result. On $\R^3=(x_1,x_2,x_3)$, let $X_1$ and $X_2$ be the vector fields defined by
\begin{equation}
	\label{eq:horvf}
	X_1(x):=\d_{x_1}, \quad 
	X_2(x):=\d_{x_2} + P(x)^2\d_{x_3},
\end{equation}
where 
\begin{equation}
	\label{eq:P}
	P(x)=x_1^a-x_2^b, \quad a,b\in\N_{\geq1},
\end{equation}
and let $(\R^3,\D,g)$ be the sub-Riemannian manifold where $\D=\vspan\{X_1,X_2\}$ and $g$ is the metric making $X_1$ and $X_2$ orthonormal. Let also consider, for $\e>0$, the curve
\begin{equation}
	\label{eq:gammaintro}
	 \g_\e^{\pm,\pm}:[0,\e]\to\R^3, \quad 
	 \g_\e^{\pm,\pm}(t)=
	 \begin{cases}
	 	(\pm t^{q},\pm t,0), \quad b\geq a,\\
	 	(\pm t,\pm t^{q}, 0), \quad a\geq b,
	 \end{cases}
	 \quad q:=\frac{\max\{a,b\}}{\min\{a,b\}}.
\end{equation}
We simply denote by $\g_\e$ the curve $\g_\e^{+,+}$. If $a$ and $b$ are not multiples, every curve $\g_\e^{\pm,\pm}$ is of class $C^{\lfloor q\rfloor}\setminus C^{\lfloor q\rfloor+1}$ at $0$ in its arc-length reparameterization.

\begin{theorem}
	\label{thm:main1}
	If $a=2$ and $b\geq5$ is odd, then there exists $\e>0$ such that the curve $\g_\e$ is length-minimizing in $(\R^3,\D,g)$. Moreover, the curve $\g_\e$ is non-smooth at 0.
\end{theorem}

Theorem \ref{thm:main1} is proved in \cite{CJMRSSS24} and provides the first ever examples of non-smooth sub-Riemannian length-minimizing curves, answering a longstanding open question in sub-Riemannian geometry. The geodesic with the lowest regularity is of class $C^2\setminus C^3$ and it is obtained for $a=2$ and $b=5$ (the singularity is at a boundary point).

In this paper, we delve into the reasoning and methodology behind the construction of the examples in \eqref{eq:horvf}-\eqref{eq:P} and we prove that Theorem \ref{thm:main1} is sharp, in the sense that it is not possible to find a geodesic of class $C^1\setminus C^2$ within these examples. Indeed, in this setting, every length minimizer is at least of class $C^1$ by \cite[Section 4]{LPS24}. Our main result is the following.

\begin{theorem}
	\label{thm:main2}
	If $b>a\geq3$, or $a=2$ and $b=3$, or $a>b\geq3$, or $b=2$ and $a=3$, then, for every $\e>0$, every curve $\g_\e^{\pm,\pm}$ is not length-minimizing.
\end{theorem}

Since every $C^1$ but non-smooth abnormal curve in $(\R^3,\D)$ must contain $\g_\e$, $\g_\e^{+,-}$, $\g_\e^{-,+}$, or $\g_\e^{-,-}$ as a subarc (this is a consequence of \eqref{eq:mart_ex} and Lemma \ref{lem:Martinet} below, and of the arguments used in \cite[Section 4]{LPS24}), we have the following consequence.

\begin{corollary}
	For all $a,b\in\N_{\geq1}$, every geodesic in $(\R^3,\D,g)$ is of class $C^2$.
\end{corollary}

{ Another consequence of Theorem \ref{thm:main1} is the existence of $C^1$-rigid curves that are not length-minimizing, a problem that has been pointed out by Montgomenry in his book \cite[Chapter 3]{Mon02}. A curve is $C^1$-rigid if it admits a $C^1$ neighborhood containing no other horizontal curves (up to reparameterization) with the same endpoints. In Lemma \ref{lem:C1-rigid} we prove that, for every $b>a$, the curve $\g_\e$ is $C^1$-rigid. As a consequence, we have the following result.

\begin{corollary}
	\label{cor:C1rigid-nonmin}
	For all $b>a\geq3$ and for $(a,b)=(2,3)$, the curve $\g_\e$ is a strictly abnormal and $C^1$-rigid curve which is not length-minimizing.
\end{corollary}

}

The proof of Theorem \ref{thm:main2} is constructive, as it provides an explicit example of a shorter curve connecting the same points. This kind of construction is, in our opinion, a general interesting tool to understand (as a first analysis) the minimality of abnormal curves. Indeed, in these examples, the shorter curves we construct are a close approximation of normal curves, see~\cite[Proposition 2.2]{CJMRSSS24}. On the other hand, the proof makes explicit the analytical role of the parameters $a$ and $b$, which are in fact the discriminant for both the minimality and the regularity of the curve $\g_\e$. Instead, the intrinsic and geometric nature of these parameters persists as an enigma, and we believe that understanding it is a deep and central question. Finding an abstract argument that unifies all these points --the explicit construction in the proof of Theorem \ref{thm:main2}, approximation of normal curves, and the interplay between analytical and geometric aspects of the parameters $a$ and $b$-- would represent, in our opinion, a significant breakthrough for the theory, which may leads either to find an example of a $C^1\setminus C^2$ length-minimizing curve, or to prove the $C^2$ (or $C^1$) regularity, or to find an example of a non-smooth geodesic with internal singular point. Ultimately, Theorems \ref{thm:main1} and \ref{thm:main2} may open new frontiers of research in sub-Riemannian geometry.

\medskip

We conclude this introductory part with an outline of the paper.

In Section~\ref{sec:notation}, we introduce a class of sub-Riemannian structures on $\R^3$ that generalizes those defined in \eqref{eq:horvf}-\eqref{eq:P}, see Definition~\ref{def:general_setting}, and we study their general properties, including a geometric characterization of abnormal curves in Lemma \ref{lem:Martinet}. We also provide some concrete examples and we explain the ideas behind the construction of the examples in~\eqref{eq:horvf}-\eqref{eq:P}, whose structure (in particular the square of $P$) is inspired by the Liu-Sussmann example in~\cite{LS95}.

In Section 3, we prove Theorem~\ref{thm:main2}. The proof is written only for $\g_\e$, since the other cases for $\g_\e^{+,-}$, $\g_\e^{+,-}$, and $\g_\e^{-,-}$ are identical. The proof relies upon the construction of a competing curve which is shorter than $\g_\e$. Since every horizontal curve is implicitly defined by its first two coordinates, the construction of the competitor can be reduced to a constrained problem in $\R^2$, see Definition~\ref{def:plane_comp}. The constraint we obtain is, due to the Stokes theorem, of isoperimetric type with a weighted area, see Lemma~\ref{lem:Stokes}. The competing curve (actually, its projection onto the $(x_1,x_2)$-plane) is then built as follows: we cut $\g_\e$ near the origin with a segment connecting 0 to the point $\g_\e(\rho)$, for some $0<\rho<\e$; then we follow the curve $\g_\e$ up to its end-point; finally, we add the boundary of a square of length $0<\de<\e$ at the end-point of $\g_\e$. Thanks to Remark~\ref{rem:eucl}, the sub-Riemannian length of both $\g_\e$ and the competitor is measured by the Euclidean length of their projection on the $(x_1,x_2)$-plane. Therefore, the cut produces a gain of length $\De L(\rho)$, but it modifies the third coordinate. The square is then needed to restore the third coordinate to 0, despite it wastes an amount of length equal to $\de$. The proof consists in showing that the parameters $\rho,\de>0$ can be chosen such that the length gained with the cut is greater than that lost with the boundary of the square, i.e., such that $\De L(\rho) > \de$.

\medskip
{\bf Research funding.} 
This project has received funding from the European Union’s Horizon 2020 research and innovation programme under the Marie Sk{\l}odowska-Curie grant agreement No 101034255. \includegraphics[width=0.65cm]{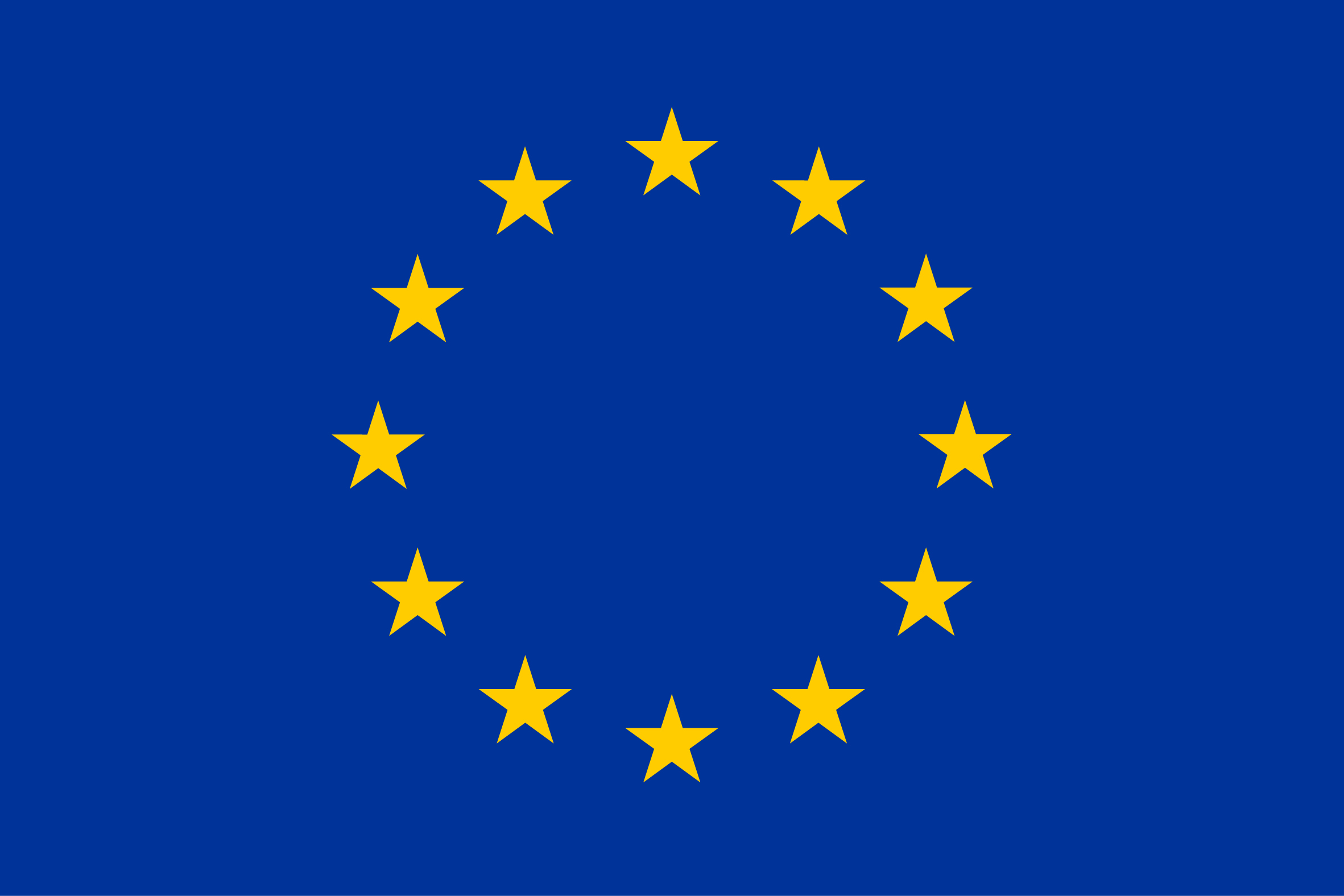}

This project has received funding from the Fondazione Ing. Aldo Gini of Padova, Italy.

This project benefited from the support of the FMJH Program PGMO.





\section{Sub-Riemannian structures on $\R^3$}
\label{sec:notation}

In this section, we study the class of sub-Riemannian structures on $\R^3$ defined in Definition \ref{def:general_setting} below and their general properties, we provide some examples, and we explain the ideas behind the construction of the structure defined in \eqref{eq:horvf}-\eqref{eq:P}.

\begin{definition}
	\label{def:general_setting}
	On $\R^3$ with the coordinates $x=(x_1,x_2,x_3)$ we define the vector fields
	\begin{equation}
		X_1(x)=\frac{\d}{\d x_1}, \quad X_2(x)=\varphi(x)\frac{\d}{\d x_2} + \psi(x)\frac{\d}{\d x_3},
	\end{equation}
	where $\varphi,\psi\in C^\oo(\R^3,\R)$ are such that $X_1$ and $X_2$ satisfy the H\"ormander condition. We also define the distribution $\D := \vspan\{X_1, X_2\}$ and the sub-Riemannian metric $g$ that makes $X_1$ and $X_2$ orthonormal.
\end{definition}

\begin{definition}
An absolutely continuous curve $\eta:[0,\tau]\to\R^3$ is {\em admissible} or {\em horizontal} in $(\R^3,\D)$ if there exists a function $u=(u_1,u_2)\in L^2([0,\tau],\R^2)$, called the {\em control} of $\eta$, such that
\begin{equation}
	\label{eq:doteta}
	\dot\eta(t)=u_1X_1(\eta(t))+u_2X_2(\eta(t)), \quad \text{for a.e. } t\in[0,\tau].
\end{equation} 
\end{definition}

	By~\eqref{eq:doteta} it follows that a horizontal curve $\eta:[0,\tau]\to\R^3$ and its control $u$ satisfy, for almost every $t\in[0,\tau]$,
	\begin{equation}
		\dot\eta_1(t)=u_1(t), \quad \dot\eta_2(t)=u_2(t)\varphi(\eta(t)), \quad \dot\eta_3(t)=u_2(t)\psi(\eta(t)).
	\end{equation}
	If $\varphi\circ\eta\neq0$, then we have $\dsy u_2(t)=\frac{\dot\eta_2(t)}{\varphi(\eta(t))}$, for a.e. $t\in[0,\tau]$, and thus the third coordinate of $\eta$ is uniquely determined by the first two coordinates of $\eta$ via the integral constraint
	\begin{equation}
		\label{eq:3coord}
		\eta_3(t)=\eta_3(0)+\int_0^t \dot\eta_2(s)\frac{\psi(\eta(s))}{\varphi(\eta(s))}ds, \quad \text{for all $t\in[0,\tau]$}.
	\end{equation}

\begin{definition}
	Let $\eta:[0,\tau]\to\R^3$ be a horizontal curve in $(\R^3,\D,g)$ and let $x,y\in\R^3$ be two points. We say that:
	\begin{itemize}
		\label{def:sRlen}
		\item[(i)] the sub-Riemannian length of $\eta$ is
		\begin{equation}
			\label{eq:deflen}
			L_{\mr{sR}}(\eta):=\int_0^\tau \sqrt{g\big(\dot\eta(t),\dot\eta(t)\big)}\,dt;
		\end{equation}
	
		\item[(ii)] the sub-Riemannian distance between $x$ and $y$ is
		\begin{equation}
			d(x,y):=\inf\{L_{\mr{sR}}(\wt\eta)\mid \wt\eta :[0,\tau]\to\R^3 \text{ is horizontal}, \; \wt\eta(0)=x, \; \wt\eta(\tau)=y\};
		\end{equation}
	
		\item[(iii)] $\eta$ is length-minimizing if for every horizontal curve $\wt\eta:[0,\tau]\to\R^3$ such that $\wt\eta(0)=\eta(0)$ and $\wt\eta(\tau)=\eta(\tau)$, we have $L_{\mr{sR}}(\eta)\leq L_{\mr{sR}}(\wt\eta)$.
	\end{itemize}
\end{definition}
	
The fact that the distance in (ii) is well defined (i.e. $d(x,y)<+\oo$ for all $x,y\in\R^3$) is a direct consequence of the H\"ormander condition and it is well-known as Chow-Rashewskii theorem, see for instance~\cite{ABB20}. Consequently, also the definition of length-minimality in (iii) is well posed. 

\begin{remark}
	\label{rem:eucl}
	If $\varphi\equiv1$ every horizontal curve $\eta$ has control $u=(\dot\eta_1,\dot\eta_2)$. In this case, since $g$ is the metric making $X_1,X_2$ orthonormal, we have that the sub-Riemannian length of $\eta$ is the Euclidean length of the plane curve $(\eta_1,\eta_2)$.
\end{remark}
	
In the literature, necessary conditions (of the first order) for horizontal curves to be length-minimizing can be obtained in two different ways. The first possible approach is to study the first differential of the end-point map together with the Lagrange multiplier rules. The other way is to use the Pontryagin maximum principle provided by classical control theory. Both the methods are well illustrated in~\cite{ABB20, AgrSac}.
	
\begin{proposition}
	\label{prop:PMP}
	Let $\eta:[0,\tau]\to\R^3$ be a length-minimizing curve in $(\R^3,\D,g)$. Then there exists an absolutely continuous function of covectors $p:[0,\tau]\to(\R^3)^*$ such that, for $i=1,2$ and for all $t\in[0,\tau]$, one of the following conditions holds
	\begin{align}
		&\label{eq:normal} p(t)\big(X_i(\eta(t))\big)=g\big(\dot\eta(t),X_i(\eta(t))\big),\\
		&\label{eq:abnormal}p(t)\big(X_i(\eta(t))\big)=0, \quad \text{with } p(t)\neq0.
	\end{align}
\end{proposition}

If $\eta:[0,\tau]\to\R^3$ is horizontal and there exists a function of covectors $p:[0,\tau]\to(\R^3)^*$ such that~\eqref{eq:normal} or~\eqref{eq:abnormal} holds, then $\eta$ is called an {\em extremal curve}. Extremal curves satisfying~\eqref{eq:normal} are called {\em normal}, while extremal curves satisfying~\eqref{eq:abnormal} are called {\em abnormal}. The conditions in~\eqref{eq:normal}-\eqref{eq:abnormal} are mutually exclusive only if the function of covectors $p$ is fixed, i.e., an extremal curve can be both normal and abnormal with different covectors. A curve that is only normal (resp., only abnormal) is called {\em strictly normal} (resp., {\em strictly abnormal}). 

We briefly recall that normal curves are solution to the normal Hamiltonian equations. Here and hereafter, we use the short notation $\varphi_{x_i},\psi_{x_i}$ to denote the partial derivatives of the functions $\varphi$ and $\psi$ (see Definition \ref{def:general_setting}) with reaspect to $x_i$, $i=1,2,3$. Defining the normal Hamiltonian of $(\R^3,\D,g)$ as
\begin{equation}
	\label{eq:H}
	H:\R^3\times(\R^3)^*\to \R, \quad H(x,p):= \frac12 p_1^2 + \frac12 h(x,p)^2,
\end{equation} 
where $h(x,p):=p_2\varphi(x)+p_3\psi(x)$, we have that normal curves with their covectors are solution to the system of equations
\begin{equation}
	\label{eq:Hameqs}
	\begin{cases}
		\dot x_1=p_1\\
		\dot x_2=h(x,p)\varphi(x)\\
		\dot x_3=h(x,p)\psi(x),
	\end{cases}
	\begin{cases}
		\dot p_1=-h(x,p)(p_2\varphi_{x_1}(x)+p_3\psi_{x_1}(x))\\
		\dot p_2=-h(x,p)(p_2\varphi_{x_2}(x)+p_3\psi_{x_2}(x))\\
		\dot p_3=-h(x,p)(p_2\varphi_{x_3}(x)+p_3\psi_{x_3}(x)).
	\end{cases}
\end{equation}
We refer the reader to~\cite{ABB20, Rif14} for all these facts, including a proof of~\eqref{eq:Hameqs}.

We now focus on the study of the abnormal curves in $(\R^3,\D)$, aiming to provide their geometric characterization. As a matter of fact, the abnormal condition~\eqref{eq:abnormal} does not depend on the metric. We start by recalling that necessary conditions of the second order for abnormal curves to be length-minimizing, known as Goh conditions, can be obtained by the study of the second-order differential of the end-point map. Moreover, the optimality  assumption is unnecessary when $\dim(\D_x)=2$, for all $x\in\R^3$. Thus, every abnormal curve $\eta$ in $(\R^3,\D)$ with its function of covectors $p$ satisfies
\begin{equation}
	\label{eq:goh}
	p(t)([X_1,X_2](\eta(t)))=0, \quad \mr{for \; all} \; t\in[0,\tau],
\end{equation}
where $[\cdot,\cdot]$ denotes the Lie bracket of vector fields. For more details about Goh conditions and for a proof of~\eqref{eq:goh} we refer the reader to~\cite{ABB20}.

Combining~\eqref{eq:abnormal} and~\eqref{eq:goh} we deduce the desired geometric characterization for abnormal curves in $(\R^3,\D)$.

\begin{lemma}
	\label{lem:Martinet}
	A horizontal curve $\eta:[0,\tau]\to\R^3$ is abnormal in $(\R^3,\D)$ if and only if it is supported inside the surface
	\begin{equation}
		\label{eq:martinet_surf}
		\Sigma_\D:=\{x\in\R^3\mid \varphi(x)\psi_{x_1}(x)-\psi(x)\varphi_{x_1}(x)=0\}.
	\end{equation}
\end{lemma}

\begin{proof}
Let $\eta:[0,\tau]\to\R^3$ be an abnormal curve and let $p=(p_1,p_2,p_3):[0,\tau]\to(\R^3)^*$ be its covector. By~\eqref{eq:abnormal} with $i=1$ we get $p_1=0$. Then by~\eqref{eq:abnormal} with $i=2$ and by~\eqref{eq:goh} we have that $\eta$ satisfies, for all $t\in[0,\tau]$,
\begin{equation}
	\label{eq:system_martinet}
	\begin{cases}
		p_2(t)\varphi(\eta(t)) + p_3(t)\psi(\eta(t))=0,\\
		p_2(t)\varphi_{x_1}(\eta(t)) + p_3(t)\psi_{x_1}(\eta(t))=0.
	\end{cases}
\end{equation}
Since $(p_2(t),p_3(t))\neq 0$, we have that for all $t\in[0,\tau]$ the determinant of the matrix of the system~\eqref{eq:system_martinet} is forced to be 0, that is
\begin{equation}
	\label{eq:det=0}
	(\varphi\psi_{x_1}-\psi\varphi_{x_1})(\eta(t))=0.
\end{equation}

Conversely, if $\eta$ is supported in $\Sigma_\D$ then it holds~\eqref{eq:det=0}. Thus for every $t\in[0,\tau]$ there exists $(p_2(t),p_3(t))\neq0$ such that~\eqref{eq:system_martinet} holds. Consequently, the covector $p:[0,\tau]\ni t\mapsto(0,p_2(t),p_3(t))\in(\R^3)^*$ ensures that $\eta$ is abnormal.
\end{proof}

\begin{definition}
	The set $\Sigma_\D$ defined in~\eqref{eq:martinet_surf} is the {\em Martinet surface} of $\D$.
\end{definition}

We conclude this section providing some examples of distributions of the type appearing in Definition~\ref{def:general_setting}, which vary depending on the functions $\varphi$ and $\psi$. The aim is to study Martinet surfaces and abnormal curves within these examples. In particular, we mainly focus on the case where $\varphi,\psi$ are chosen as in~\eqref{eq:horvf}-\eqref{eq:P}, and we explain through an heuristic argument why the curve $\g_\e=\g_\e^{+,+}$ defined in~\eqref{eq:gammaintro} emerges as a natural candidate to be a non-smooth geodesic for every $b>a\geq2$ and for small $\e>0$.

\medskip

\noindent{\bf The Heisenberg group.} This represents the simplest and most famous example of a sub-Riemannian manifold. It is obtained by setting $\varphi=1$ and $\psi=x_1$. In this example, the Martinet surface is empty, so there are no abnormal curves. The Heisenberg group has been extensively studied in the literature, and interested readers are directed for example to~\cite{Mon02}.

\medskip

\noindent{\bf The Martinet example.} This example involves a slight modification of the Heisenberg group by choosing $\varphi=1$ and $\psi=x_1^2$. Consequently, the Martinet surface is no longer empty and is given by the plane $x_1=0$. Thus, for $\e>0$, the curve
\begin{equation}
	\label{eq:alfa}
	\a_\e:[0,\e]\to\R^3, \quad \a(t)=(0,t,0),
\end{equation}
is abnormal. Since $\a_\e$ is a segment in $\R^2$, by Remark~\ref{rem:eucl} it is optimal in $(\R^3,\D,g)$ for all $\e>0$, providing an example of an abnormal length-minimizing curve. However, the curve $\a_\e$ is not strictly abnormal since it is also normal in $(\R^3,\D,g)$ with the constant covector $p=(0,1,0)$.

\medskip

\noindent{\bf The Liu-Sussmann example.} This example, credited to Liu and Sussmann (see~\cite{LS95} or~\cite[Chapter 3]{Mon02}), consists of a slight modification of the Martinet example to make the curve $\alpha$ in~\eqref{eq:alfa} strictly abnormal. The choice of $\varphi$ and $\psi$ here is $\varphi=1-x_1$ and $\psi=x_1^2$. Consequently, the defining function of $\Sigma_{\mathcal{D}}$ in~\eqref{eq:martinet_surf} becomes
\begin{equation}
	2x_1(1-x_1)+x_1^2=x_1(2-x_1),
\end{equation}
implying that the Martinet surface is the union of the two planes $x_1=0$ and $x_1=2$.

The curve $\a_\e$ defined in~\eqref{eq:alfa} remains abnormal and, since $\varphi$ is no longer equal to 1, the sub-Riemannian length of horizontal curves here is no longer the Euclidean length of their projection onto $\R^2$. Therefore, the length-minimality of $\a_\e$ is no longer obvious in this example. In~\cite{LS95} (see also~\cite[Chapter 3]{Mon02} or~\cite[Section 2.5]{Rif14}) it is proved that $\a_\e$ is both strictly abnormal and length-minimizing for small $\e>0$. 

The proof of Liu and Sussmann relies on the fact that the curve $\a_\e$ is $C^1$-rigid.

\begin{definition}
	\label{def:C1-rigid}
	Let $\eta:[0,\tau]\to\R^3$ be a horizontal curve with control $u$. For $\de>0$, define 
	\begin{equation}
		\begin{split}
			\mathcal U_{\eta,\de}&:=\big\{\tilde\eta:[0,\tau]\to\R^3 \text{ horizontal with control } \tilde u \mid \\ &\qquad\;\tilde\eta(0)=\eta(0), \tilde\eta(\tau)=\eta(\tau), \text{ and }
			\|\eta-\tilde\eta\|_\oo, \|u-\tilde u\|_\oo<\de\big\}.
		\end{split}
	\end{equation}
	We say that $\eta$ is {\em $C^1$-rigid} if there is $\de>0$ such that every curve in $\mathcal U_{\eta,\de}$ is a reparameterization of $\eta$.
\end{definition}

We notice that Definition \ref{def:C1-rigid} does not depend on the parameterization of $\eta$. Indeed, if $\varphi:[0,\tau']\to[0,\tau]$ is a reparameterization of $\eta$, then for every $\de>0$ and every $\tilde \eta \in\mc U_{\eta\circ\varphi,\de}$, we have that $\tilde\eta\circ\varphi^{-1}\in \mc U_{\eta,\,\de\max\{1,\|\dot\varphi^{-1}\|_\oo\}}$.

\begin{remark}
	\label{C1iso}
	The $C^1$-rigidity of $\a_\e$ is provided by the squared power appearing on $\psi=x_1^2$, see \cite[Chapter 3]{Mon02}. As explained at the end of this section, the property of being $C^1$-rigid is not enough to be length-minimizing.
\end{remark}

\begin{remark}
	The curve $\a_\e$ remains strictly abnormal by choosing $\psi=x_1^k$ for every integer $k\geq2$. However, it remains $C^1$-rigid only for even $k$ and, in fact, the proof of Liu and Sussmann showing the minimality of $\a_\e$ can be repeated only if $k$ is even. Instead, for odd $k$, the curve $\a_\e$ is not of minimal length, see~\cite[Theorem 10.1]{BMS24}.
\end{remark}

\medskip

\noindent{\bf The examples defined in \eqref{eq:horvf}-\eqref{eq:P}.}
Our examples arise as a different modification of the Martinet example than the one as in the Liu-Sussmann case. The aim is to produce a singularity on the Martinet surface, then generating non-smooth and possibly lenght-minimizing abnormal curves.

We choose $\varphi=1$ and $\psi=P(x)^2$, where $P(x)=x_1^a-x_2^b$ and $a,b\in\N\setminus\{0\}$. The choice of $\varphi=1$ ensures, according to~\eqref{eq:martinet_surf}, that possible singularities on $\Sigma_\D$ only depend on $\psi$. With respect to the Liu-Sussmann example, we also recover the property that the sub-Riemannian length of horizontal curves is the Euclidean length of their projections onto $\R^2$, see Remark~\ref{rem:eucl}.

The choice of $\psi=P^2$ is then thought of as a perturbation of the square $(x_1^a)^2$, generating the desired singularity. The squared power on $P$ aims to maintain a similar analytic structure as in the Liu-Sussmann example and its importance is better explained in Remarks~\ref{rem:P^2_1} and~\ref{rem:P^2_2} below.

By~\eqref{eq:martinet_surf}, the Martinet surface is given by the zero locus of the derivative of $P^2$ with respect to $x_1$. We define the abnormal polynomial
\begin{equation}
	\label{eq:A}
	Q(x_1,x_2):=\d_{x_1}\big(P(x_1,x_2)^2\big)=2ax_1^{a-1}P(x_1,x_2),
\end{equation}
and so the Martinet surface is
\begin{equation}
	\label{eq:mart_ex}
	\Sigma_\D=\{x\in\R^3\mid Q(x_1,x_2)=0\}.
\end{equation}
If $a=1$ we have $\Sigma_\D=\{x\in\R^3\mid P=0\}$, while for $a\geq2$ we have
\begin{equation}
	\label{eq:mart_ex2}
	\Sigma_\D=\{x\in\R^3\mid x_1=0\}\cup\{x\in\R^3\mid P(x_1,x_2)=0\}.
\end{equation}

\begin{remark}
	\label{rem:P^2_1}
	Thanks to the squared power on $P$, the zero locus of $P$ is part of the Martinet surface for every $a\geq1$. Therefore, the curve $\g_\e=\g_\e^{+,+}$ defined in \eqref{eq:gammaintro} is abnormal in $(\R^3,\D)$, for every $\e>0$ and for every $a$ and $b$. Moreover, the curve $\g_\e$ is strictly abnormal in $(\R^3,\D,g)$ for $a\neq b$, as showed in the next lemma. When $a$ and $b$ are not multiples $\g_\e$ is not smooth at 0, motivating the interest in these examples.
\end{remark}

\begin{lemma}
	\label{lem:str_abn}
	Let $a,b\in\N\setminus\{0\}$ with $a\neq b$. Then, no parameterization of the curve $\g_\e$ defined in~\eqref{eq:gammaintro} is normal in $(\R^3,\D,g)$, for any $\e>0$.
\end{lemma}

\begin{proof}
	Assume by contradiction that there exists a parameterization $\bar\g:[0,1]\to\R^3$ of $\g_\e$ and a function of covectors $p:[0,1]\to(\R^3)^*$ such that $\bar\g$ and $p$ are solution of~\eqref{eq:Hameqs}. Since $P(\bar\g(t))=0$ for all $t\in[0,1]$, the equations in~\eqref{eq:Hameqs} read
	\begin{equation}
		\begin{cases}
			\dot {\bar\g}_1=p_1\\
			\dot {\bar\g}_2=p_2\\
			\dot {\bar\g}_3=0,
		\end{cases}
		\begin{cases}
			\dot p_1=0\\
			\dot p_2=0\\
			\dot p_3=0.
		\end{cases}
	\end{equation} 
	Then both $p_1=p_1(0)$ and $p_2=p_2(0)$ are constants. Moreover, we have $\bar\g_1=p_1(0)t$ and $\bar\g_2=p_2(0)t$. Plugging these identities into $P(\bar\g)=0$ we get
	\begin{equation}
		\label{eq:non-normal}
		p_1(0)^at^a - p_2(0)^bt^b=0, \quad \text{for all } t\in[0,1],
	\end{equation}
	which is not possible for $a\neq b$, unless $p(0)=0$.
\end{proof}

The second crucial role of the squared power on $P$ is contained in the next lemma and explained in Remark \ref{rem:P^2_2} below.

\begin{lemma}
	\label{lem:C1-rigid}
	For all $b>a$ and $\e>0$, the curve $\g_\e$ defined in \eqref{eq:gammaintro} is $C^1$-rigid.
\end{lemma}

\begin{proof}
	We recall that $\g_\e:[0,\e]\to \R^3$ is defined by $\g_\e(t)=(t^q,t,0)$, $q=\frac ba$.
	
	Assume by contradiction that for all $\de>0$ there is $\tilde\eta\in\mc U_{\g_\e,\de}$ which is not a reparameterization of $\g_\e$. Denote by $\tilde u$ the control of $\tilde\eta$. Since the control of $\g_\e$ is $(qt^{q-1},1)$, $t\in[0,\e]$, we have that
	\begin{equation}
		|1-\tilde u_2(t)|<\de, \text{ for a.e. } t\in[0,\e].
	\end{equation}
	Fix $\de\in(0,\frac12]$. Then we have
	\begin{equation}
		\label{eq:control-C1rigid}
		\tilde u_2(t)>1-\de\geq \frac12, \text{ for a.e. } t\in[0,\e].
	\end{equation}
	
	On the other hand, since $\tilde \eta\in\mc U_{\g_\e,\de}$, we have that $\tilde\eta(0)=0$ and $\tilde\eta(\e)=(\e^q,\e,0)$. In particular, by \eqref{eq:3coord} and \eqref{eq:control-C1rigid}, we deduce
	\begin{equation}
		\label{eq:thanks_P2}
		0=\tilde\eta_3(\e)=\int_0^\e \tilde u_2(t) P\big(\tilde \eta_1(t),\tilde \eta_2(t)\big)^2dt\geq \frac12 \int_0^\e P\big(\tilde \eta_1(t),\tilde \eta_2(t)\big)^2dt.
	\end{equation}
	The last inequality implies $P\big(\tilde\eta_1(t),\tilde \eta_2(t)\big)=0$, for all $t\in[0,\e]$. Consequently, $\tilde\eta$ is a reparameterization of $\g_\e$, which is a contradiction.
\end{proof}

\begin{remark}
	\label{rem:P^2_2}
	Without an even power on $P$, it is not possible to conclude that $P(\tilde\eta_1,\tilde\eta_2)=0$ from \eqref{eq:thanks_P2}. Thus, the squared power on $P$ ensures the $C^1$-rigidity of $\g_\e$, a property that Montgomery highlighted as a potentially sufficient condition to provide length-minimality, see \cite[Chapter 3]{Mon02}. 
	Also, for $b>a\geq2$ and for small $\e>0$, the curve $\g_\e$ is close to the curve $\a_\e$ defined in \eqref{eq:alfa} in the $C^1$ topology, being the two curves tangent at the origin (when $a>b\geq2$, the two curves are instead orthogonal at 0). Since the Euclidean metric on $\R^2$ measures the length of horizontal curves, the curve $\a_\e$ is length-minimizing (as in the Martinet case, see Remark \ref{rem:eucl}). Definitely, the squared power on $P$ creates an analytic structure similar to the Martinet and the Liu-Sussmann ones, where the curve $\g_\e$ is, at least heuristically, a very good candidate to be lenght-minimizing.
\end{remark}

The proof of Theorem~\ref{thm:main2} that we give in the next section shows that the parameters $a$ and $b$ partially destroy the heuristic intuition behind the construction of these examples. Their role seems to be much deeper than the analytic one they carry out in the computations. Also, Theorem \ref{thm:main2} and Lemma \ref{lem:C1-rigid} prove Corollary \ref{cor:C1rigid-nonmin}, showing that the property of being $C^1$-rigid is not enough to deduce that a curve is length-minimizing.

\section{Proof of Theorem~\ref{thm:main2}}

\label{sec:pfthm2}

The proof relies upon the construction of a shorter competitor. We consider the sub-Riemannian manifold $(\R^3,\D,g)$ and the strictly abnormal curve $\g_\e=\g_\e^{+,+}$, $\e>0$, defined in \eqref{eq:horvf}, \eqref{eq:P}, and \eqref{eq:gammaintro}. The proof that also $\g_\e^{+,-}$, $\g_\e^{-,+}$, and $\g_\e^{-,-}$ are not length-minimizing is identical. 

From now on, $\e>0$ is fixed. We also introduce the useful notation
\begin{equation}
	q:=\frac{\max\{a,b\}}{\min\{a,b\}}.
\end{equation}

We prove the theorem for $b>a$, thus we need to prove that $\g_\e$ is not of minimal length for $b>a\geq3$, or for $a=2$ and $b=3$, and for every $\e>0$. Here we have $q=\frac ba$. Then, in Remark \ref{rem:a>b}, we explain how to adapt the proof to the case $a>b$ (the two cases are essentially identical).

\medskip

The starting point to prove Theorem~\ref{thm:main2} is to substitute competing curves for $\g_\e$, i.e., horizontal curves having the same end-points of $\g_\e$, with plane curves satisfying a certain constraint. The plane curve is the projection onto $\R^2$ of the competitor and the constraint to which it is subjected is the one in \eqref{eq:3coord} defining the third coordinate. 

\begin{definition}
	\label{def:plane_comp}
	For an absolutely continuous curve $\w:[0,\tau]\to\R^2$, we define the integral function $I_\w:[0,\tau]\to\R$,
	\begin{equation}
		\label{eq:dotomega3}
		I_\w(t):=\int_{0}^{t} \dot\w_2(s)P(\w(s))^2ds.
	\end{equation}
	We say that $\w$ is a {\em competing curve for $\g_\e$}, or simply a {\em competitor}, if $\w(0)=0$, $\w(\tau)=(\e^q,\e)$, and $I_\w(\tau)=0$.	
\end{definition}

In the following, we denote by $\mc C_\e([0,\tau],\R^2)$ the set of competitors parameterized in the interval $[0,\tau]$. For a competitor $\w\in \mc C_\e([0,\tau],\R^2)$, we call the curve $(\w_1,\w_2,I_\w):[0,\tau]\to\R^3$ the horizontal lift of $\w$. Also, with abuse of notation, we denote by $\g_\e$ the projection on $\R^2$ of $\g_\e$ itself, according to the fact that $\g_3$ is constantly 0.

\medskip

\newcommand{\ind}{{\mathrm{ind}}}

We next observe a consequence of Stokes theorem, providing that for all $\w\in \mc C_\e([0,\tau],\R^2)$ the condition $I_\w(\tau)=0$ in Definition~\ref{def:plane_comp} is a constraint of isoperimetric nature, involving some weighted area.

Let $\vartheta_i:[0,\tau_i]\to\R^2$, $\tau_i>0$, $i=1,2$, be two absolutely continuous plane curves. If $\vartheta_1(\tau_1)=\vartheta_2(0)$, we say that $\vartheta_1$ and $\vartheta_2$ are {\em concatenable}, and we define their concatenation as
\begin{equation}
	(\vartheta_1*\vartheta_2)(t):=
	\begin{cases}
		\vartheta(t), \quad &t\in[0,\tau_1],\\
		\vartheta_2(t-\tau_1), \quad &t\in[\tau_1,\tau_1+\tau_2].
	\end{cases} 
\end{equation}
We also define the {\em inverse parameterization} of $\vartheta_1$ as
\begin{equation}
	-\vartheta_1(t):=\vartheta_1(\tau_1-t).
\end{equation}

If $\w\in \mc C_\e([0,\tau],\R^2)$, we denote by $\Omega=\Omega(\w,\e)  \subset \mathbb R^2$ the bounded open subset of $\mathbb R^2$ enclosed
by $\w*(-\g_\e)$. For each connected component $U\subset \Omega$, we denote by $\mathrm{ind}(U)\in\mathbb Z$ its winding number with respect to the 
curve $\w*(-\g_\e)$. Then we define the weighted area of $U$ as 
\begin{equation}
	\label{eq:areapesata}
	\mathscr L_Q(U):=\ind(U)\iint_U Q(x_1,x_2)dx_1dx_2,
\end{equation}
where $Q$ is the abnormal polynomial in~\eqref{eq:A}.

\begin{lemma}
	\label{lem:Stokes}
	For any $\omega \in \mathcal C_\e([0,\tau],\mathbb R^2)$ we have 
	\begin{equation}
		\label{eq:Stokes}
		\sum_{U} \mathscr L_Q(U)=0,
	\end{equation} 
	where the sum is over all connected components $U\subset\Omega$.
\end{lemma}

\begin{proof}
	Since $\w\in \mc C_\e([0,\tau],\mathbb R^2)$ we have $I_\w(\tau)=0$ (see Definition~\ref{def:plane_comp}). By~\eqref{eq:dotomega3} this condition reads
	\begin{equation}
		0=\int_0^\tau \dot\w_2(t) P(\w(t))^2dt=\int_{\w}P(x_1,x_2)^2dx_2.
	\end{equation}
	Since $P(\g_\e)$ is identically 0, from the last equation we deduce 
	\begin{equation}
		0=\int_{\w}P(x_1,x_2)^2dx_2+\int_{-\g_\e}P(x_1,x_2)^2dx_2=\int_{\w*(-\g_\e)}P(x_1,x_2)^2dx_2.
	\end{equation}
	The statement then follows by Stokes Theorem.
\end{proof}

We are now ready to construct the competitor and to prove that it is shorter than $\g_\e$ if $b>a\geq3$, or if $a=2$ and $b=3$ (we recall that we are proving the theorem when $b>a$ and that we explain how to adapt the proof to the case $a>b$ in Remark~\ref{rem:a>b}). 

The competitor is built in the following way (see also the figure below):
\begin{itemize}
	\item[i)] we cut $\g_\e$ with a segment near the origin. For $0<\rho<\e$, we define the cutting segment $\k_\rho(t):=(t\rho^{q-1},t)$, $t\in[0,\rho]$, where $q=\frac ba$ and the parameter $\rho$ will be chosen later;
	\item[ii)] we continue from the point $(\rho^q,\rho)$ by following $\g_\e$ until its end point, i.e., we run across the curve $\g_\e|_{[\rho,\e]}$;
	\item[iii)] finally, we add a small rectangle at the end-point and we run across its boundary, namely, for $0<\de<\e$, we let $E_\de:=(\e^q,\e)+[0,\e^q\de]\times[-\e\de,0]$ and we let $\s=\s_\de:[0,2(\e+\e^q)\de]\to\R^2$ to be the parameterization of $\d E_\de$ given by
	\begin{equation}
		\s_\de(t):=(\e^q,\e)+
		\begin{cases}
			\s_\de(t)= (t,0), \quad &t\in[0,\e^q\de],\\
			\s_\de(t)= (\e^q\de,\e^q\de-t), \quad &t\in[\e^q\de,(\e+\e^q)\de],\\
			\s_\de(t)= ((\e+2\e^q)\de-t,-\de\e), \quad &t\in[(\e+\e^q)\de,(\e+2\e^q)\de],\\
			\s_\de(t)= (0,t-2(\e+\e^q)\de), \quad & t\in[(\e+2\e^q)\de,2(\e+\e^q)\de].
		\end{cases}
	\end{equation} 
\end{itemize}

\begin{figure*}[ht]
\centering
\begin{tikzpicture}

\def\eps{0.9}
\def\rval{0.45}
\def\qval{2}
\def\dval{0.3}

\begin{axis}[
    axis lines=middle,
    xmin=-0.02, xmax=1.2,
    ymin=-0.02, ymax=1.2,
    xlabel={$x_1$},
    ylabel={$x_2$},
    ticks=none,
    width=9cm,
    height=9cm,
    clip=false
]

\addplot[
    thick,
    black,
    domain=0:1.1,
    samples=200,
]
({x^\qval},{x});

\node[black] at (axis cs:0.4,0.7) {$\gamma$};

\addplot[
    thick,
    red,
    postaction={
        decorate,
        decoration={
            markings,
            mark=at position 0.5 with {\arrow{>}}
        }
    }
]
coordinates {(0,0) ({\rval^\qval},{\rval})};

\node[red,left] at (axis cs:0.23,0.20) {$\kappa_\rho$};

\addplot[only marks,mark=*]
coordinates {({\rval^\qval},{\rval})};

\node[above left]
at (axis cs:{\rval^\qval+0.25},{\rval-0.05})
{$(\rho^q,\rho)$};


\addplot[only marks,mark=*]
coordinates {({\eps^\qval},{\eps})};

\node[above right]
at (axis cs:{\eps^\qval-0.2},{\eps})
{$(\varepsilon^q,\varepsilon)$};

\addplot[
    thick,
    blue,
    postaction={
        decorate,
        decoration={
            markings,
            mark=at position 0.125 with {\arrow{>}},
            mark=at position 0.375 with {\arrow{>}},
            mark=at position 0.625 with {\arrow{>}},
            mark=at position 0.875 with {\arrow{>}}
        }
    }
]
coordinates {
    ({\eps^\qval},{\eps})
    ({\eps^\qval*(1+\dval*0.7)},{\eps})
    ({\eps^\qval*(1+\dval*0.7)},{\eps*(1-\dval))})
    ({\eps^\qval},{\eps*(1-\dval)})
    ({\eps^\qval},{\eps})
};

\node[blue,right]
at (axis cs:{\eps^\qval+\eps^\qval*\dval-0.05},{\eps-0.5*\eps*\dval})
{$\sigma_\delta$};

\end{axis}
\end{tikzpicture}
\end{figure*}

\begin{definition}
For $0<\rho,\de<\e$ we define the family of absolutely continuous plane curves $\w_{\rho,\de}$ as the concatenation
\begin{equation}
	\label{eq:etard}
	\w_{\rho,\de}:=\k_\rho * \g|_{[\rho,\e]} * \s_\de.
\end{equation}
\end{definition}

The curves defined in~\eqref{eq:etard} are not competitors a priori: the cut made with the curve $\k_\rho$ modifies the third coordinate of the horizontal lift of $\w_{\rho,\de}$, then we add the rectangle to correct this error, restoring the end-point of the third coordinate to 0. 

\begin{remark}
	\label{rem:win_num}
	The region enclosed by the curve $\w_{\rho,\de}*(-\g_\e)$ has two connected components, which are the rectangle $E_\de$ and the region enclosed by $\k_\rho * (-\g_\rho)$, which we call $U_{\rho}$. The winding numbers of $U_\rho$ and $E_\de$ have absolute value equal to 1, and we also have $Q|_{U_\rho}, Q|_{E_\de}\geq0$. Then the parameterization $\s_\de$ of $\d E_\de$ has been chosen in such a way that $\ind(E_\de)=-\ind(U_\rho)$, so that~\eqref{eq:Stokes} can be satisfied and $\w_{\rho,\de}$ may be a competitor.
\end{remark}

We have to chose suitable $\rho=\rho(\e)<\e$ and $\de=\de(\e)<\e$ (if it is possible) such that the curve $\w_{\rho,\de}$ both satisfies $I_{\w_{\rho,\de}}(\tau)=0$ and is shorter than $\g_\e$.

In the following, we omit the dependence on $\e$ in $\rho,\de$, and we omit the dependence on $\rho,\de$ in $\w,\k,\s, E, U$.

\begin{lemma}
	\label{lem:correction}
	For every $\e>0$ and $\rho,\de\in(0,\e)$ small enough there exist a constant $C=C(a,b,\e)>0$ and a function $f_{a,b}(\de)$, with $f_{a,b}(\de)=O(\de)$ as $\de\to0$, such that the curve $\w_{\rho,\de}$ is a competitor as soon as
	\begin{equation}
		\label{eq:correction}
		\rho^{2b+1}=C\de^3(1+f_{a,b}(\de)).
	\end{equation}
\end{lemma}

\begin{proof}
	We have to compute and then compare the effects made by the cut and by the rectangle on the third coordinate, as $\rho,\de\to0$.
	
	For the cut we have
	\begin{equation}
		\label{eq:errk}
		\begin{split}
			\De^{\cut}_3(\rho)&:=\int_{0}^\rho P(\k(t))^2\dot\k_2(t)dt=\int_{0}^{\rho} (t^a\rho^{b-a}-t^b)^2dt\\
			&\stackrel{(t=\rho s)}{=}\rho^{2b+1}\int_{0}^{1} (s^a-s^b)^2ds.
		\end{split}
	\end{equation}

	For the rectangle we use the Stokes theorem, see Lemma~\ref{lem:Stokes}. We compute
	\begin{equation}
		\label{eq:errQ'}
		\begin{split}	
			\De_3^{\cor}(\de)&:=\iint_{E_\de} Q(x_1,x_2)dx_1dx_2\\
			&=\int_{\e(1-\de)}^\e\int_{\e^q}^{\e^q(1+\de)}2ax_1^{a-1}(x_1^a-x_2^b)dx_1dx_2 \qquad (x_1=\e^qx,\; x_2=\e y)\\
			&=2a\e^{2b+1}\left(\int_{1-\de}^1\int_{1}^{1+\de} x^{a-1}(x^a-y^b)dxdy\right)
		\end{split}
	\end{equation}
    We compute the integral appearing in the last line. We have
    \begin{equation}
        \label{eq:calcoloint1}
        \begin{split}
            &\int_{1-\de}^1\int_{1}^{1+\de} x^{a-1}(x^a-y^b)dxdy 
            \\
            &= \de\frac{(1+\de)^2a-1}{2a} - \frac{(1+\de)^a-1}{a}\frac{1-(1-\de)^b}{b+1}
            \\
            &=\frac{1}{2a} \sum_{k=1}^{2a}\binom{2a}{k}\de^{k+1} - 
            \frac{1}{a(b+1)} \sum_{i=1}^{a}\sum_{j=1}^{b+1}(-1)^{j+1}\binom{a}{i}\binom{b+1}{j}\de^{i+j}
            \\
            &:= S_1 - S_2
        \end{split}
    \end{equation}
    The two sums $S_1$ and $S_2$ above reads in the following way
    \begin{align}
        \label{eq:S1}&S_1 = \de^2 + \frac{2a-1}{2}\de^3 + \frac{1}{2a} \sum_{k=3}^{2a}\binom{2a}{k}\de^{k+1}
        \\
        \label{eq:S2}&S_2 = \de^2 - \frac b2\de^3 + \frac{a-1}{2} \de^3 + \frac{1}{a(b+1)} \sum_{i=2}^{a}\sum_{j=2}^{b+1}(-1)^{j+1}\binom{a}{i}\binom{b+1}{j}\de^{i+j}
    \end{align}
    From \eqref{eq:errQ'}, \eqref{eq:calcoloint1}, \eqref{eq:S1}, and \eqref{eq:S2}, we deduce
    \begin{equation}
        \label{eq:errQ}
        \De_3^{\cor}(\de)=a(a+b)\e^{2b+1}\de^3(1+f_{a,b}(\de)),
    \end{equation}
    with
    \begin{equation}
        \label{eq:fabdelta}
        f_{a,b}(\de)= \frac{1}{2a} \sum_{k=3}^{2a}\binom{2a}{k}\de^{k-2} - \frac{1}{a(b+1)} \sum_{i=2}^{a}\sum_{j=2}^{b+1}(-1)^{j+1}\binom{a}{i}\binom{b+1}{j}\de^{i+j-3}.
    \end{equation}

	By Remark~\ref{rem:win_num}, the effects on the third coordinate produced by the cut and by the rectangle have opposite sign. Therefore, the proof of~\eqref{eq:correction} follows by~\eqref{eq:errk} and~\eqref{eq:errQ} by taking $C=C(a,b,\e)=a(b+a)\e^{2b+1}\big(\int_{0}^{1} (s^a-s^b)^2ds\big)^{-1}>0$ and $\rho,\de>0$ small enough. Moreover, by \eqref{eq:fabdelta} it follows $f_{a,b}(\de)=O(\de)$, as $\de\to0$.
\end{proof}

We are now ready to prove Theorem~\ref{thm:main2}. We denote by $L(\cdot)$ the Euclidean length of plane curves. By Remark~\ref{rem:eucl} the Euclidean lengths of $\w$ and $\g_\e$ are the sub-Riemannian lengths of their horizontal lift. The aim of the proof is to show that it is possible to choose small $\rho,\de>0$ such that both~\eqref{eq:correction} and $L(\w)<L(\g_\e)$ are satisfied.

\begin{proof}[Proof of Theorem~\ref{thm:main2}]
	Fix $\e>0$. Let $\rho,\de>0$ be small enough and such that~\eqref{eq:correction} is satisfied (we fix these parameters later). Let $C=C(a,b,\e)>0$ be the constant given by Lemma \ref{lem:correction}. We also recall the notation $q=\frac ba>1$.
	
	We start by computing the gain of length obtained with the cut, as $\rho\to0$. On the one hand, we have $\dot\g(t)=(qt^{q-1},1)$ and so
	\begin{equation}
		\label{eq:lengamma}
		\begin{split}
			L(\g_\rho)&=\int_0^\rho\sqrt{1+q^2t^{2q-2}}dt= \int_0^\rho\left(1+\frac12q^2t^{2q-2}(1+o(1))\right)dt\\
			&=\rho+\frac{q^2}{2(2q-1)}\rho^{2q-1}(1+o(1)).
		\end{split}
	\end{equation}
	On the other hand, we have
	\begin{equation}
		\label{eq:lenk}
		L(\k)=\sqrt{\rho^2+\rho^{2q}}=\rho\left(1+\frac{1}{2}\rho^{2q-2}(1+o(1))\right).
	\end{equation}
	Then, for small $\rho>0$, we get
	\begin{equation}
		\label{eq:gainlen}
		\De L(\rho) := L(\g_\rho) - L(\k) = \frac{(q-1)^2}{2(2q-1)}\rho^{2q-1}(1+o(1)) \geq 
		\frac{(q-1)^2}{4(2q-1)}\rho^{2q-1}.
	\end{equation}
	
	Since $L(\d E)=L(\s)=2\de(\e+\e^q)$, then by~\eqref{eq:gainlen} the competitor $\w$ is shorter than $\g_\e$ as soon as
	\begin{equation}
		\label{eq:upper_bound_delta}
		2\de(\e+\e^q)\leq \frac{(q-1)^2}{4(2q-1)}\rho^{2q-1}.
	\end{equation}
	Moreover, by Lemma~\ref{lem:correction}, we have 
	\begin{equation}
		\label{eq:delta>rho^2b+1/3}
		\de\geq \frac12C^{-\frac13}\rho^{\frac{2b+1}{3}}.
	\end{equation}
	By inserting~\eqref{eq:delta>rho^2b+1/3} into~\eqref{eq:upper_bound_delta}, we obtain the following sufficient condition on $\rho$ for the competitor $\w$ to be shorter than $\g_\e$:
	\begin{equation}
		\label{eq:rho_condition}
		C^{-\frac13}(\e+\e^q)\rho^{\frac{2b+1}{3}}\leq \frac{(q-1)^2}{4(2q-1)}\rho^{2q-1}.
	\end{equation}
	The latter equation is satisfied, when $\rho>0$ is small enough, if and only if
	\begin{equation}
		\frac{2b+1}{3}>2q-1,
	\end{equation}
	which in fact reads
	\begin{equation}
		\label{eq:final_ab}
		b+2>3\,\frac{b}{a}.
	\end{equation}
	If the latter equation is satisfied, then we can chose small $\rho,\de>0$ such that $\g_\e$ is not length-minimizing. Indeed, it is enough to choose $\rho>0$ small enough such that~\eqref{eq:rho_condition} is satisfied. This choice implicitly fixes $\de>0$ by~\eqref{eq:correction} and ensures $L(\w)<L(\g_\e)$. 
	
	It is very easy to see that~\eqref{eq:final_ab} is satisfied if and only if $a\geq3$, or $a=2$ and $b=3$. Since $\e>0$ was arbitrary, the proof is complete.
\end{proof}

\begin{remark}
	\label{rem:a>b}
	We explain how to adapt the proof of Theorem~\ref{thm:main2} in the case $a>b$, exploiting the symmetry within the computations between $a$ and $b$. If $a>b$ then $q=\frac ab$, and the competing curve $\w$ is essentially the same: we cut from 0 to $(\rho,\rho^{q})$ and we add the rectangle $E_\de=(\e,\e^{q})+[-\e\de,0]\times[0,\e^{q}\de]$. Again, we orient $\d E_\de$ such that $\ind(E_\de)=-\ind(U_\rho)$ (see Remark~\ref{rem:win_num}), and $\rho,\de>0$ are small parameters to be fixed at the end. 
	
	With the same computations as in~\eqref{eq:errk}-\eqref{eq:errQ} we get
	\begin{equation}
		\De_3^{\cut}(\rho)=\rho^{2a+1}\int_{0}^{1} (s^a-s^b)^2ds \quad \mr{and} \quad  \De_3^{\cor}(\de)=a(b+a)\e^{2a+q}\de^3(1+o(1)),
	\end{equation}
	where, as before, $\De_3^\cut$ is the error produced by the cut and $\De_3^{\cor}$ is the error produced by the rectangle, both in absolute value. The power of $\rho$ in $\De_3(\rho)$ replaces $b$ with $a$ (as expected due to the symmetry in $a,b$), while the slight modification in the power of $\e$ into $\De_3(\de)$ is irrelevant ($\e>0$ is fixed). Condition~\eqref{eq:correction} then implies that $\w$ is a competitor as soon as
	\begin{equation}
		\label{eq:correction'}
		\rho^{2a+1}=\bar C\de^3(1+\bar f_{a,b}(\de)),
	\end{equation}
	for some $\bar C=\bar C(a,b,\e)>0$ and $\bar f_{a,b}(\de)=O(\de)$.
	
	The computations for the gain of length in~\eqref{eq:lengamma}-\eqref{eq:lenk} are identical, then~\eqref{eq:gainlen} remains the same.

	Finally, with similar estimates as in~\eqref{eq:upper_bound_delta}-\eqref{eq:delta>rho^2b+1/3}-\eqref{eq:rho_condition}, we deduce that $\w$ is shorter than $\g_\e$ as soon as
	\begin{equation}
		\label{eq:final_ba}
		a+2>3\frac ab,
	\end{equation}
	for a suitable choice of $\rho>0$ small enough. Once a small $\rho>0$ is fixed, then also the small $\de>0$ is automatically determined. The conclusion is symmetric in $a,b$ with respect to the one obtained in the proof of Theorem~\ref{thm:main2}.
\end{remark}

\begin{remark}
	\label{rem:add_c}	
	The statement of Theorem~\ref{thm:main2} is stable with respect to arbitrary increases in the power of the polynomial $P$ defining $\psi$. That is, the statement remains true for every choice of $\psi$ of the form $\psi=(x_1^a-x_2^b)^c$, for every natural numbers $a,b,c\geq2$. Indeed, by repeating the same computations as in the proof Theorem~\ref{thm:main2} and in Remark~\ref{rem:a>b}, we obtain, in place of~\eqref{eq:final_ab}-\eqref{eq:final_ba}, that $\w$ is shorter than $\g_\e$ as soon as
	\begin{equation}
		\label{eq:final_abc}
		\frac c2\max\{a,b\}+2 > 3q.
	\end{equation}
	Since $c\geq2$ the latter equation is implied by~\eqref{eq:final_ab} when $b>a$, and by~\eqref{eq:final_ba} if $a>b$.
\end{remark}

\bibliographystyle{plain}

\bibliography{Biblio}

\end{document}